\documentclass[11pt]{article}
\usepackage{amsmath,amssymb,amsthm,graphicx}
\usepackage[mathscr]{eucal}
\usepackage[cm]{fullpage}
\usepackage{color}
\usepackage[english]{babel}
\usepackage[latin1]{inputenc}

\newtheorem{theorem}{Theorem}[section]
\newtheorem{corollary}[theorem]{Corollary}
\newtheorem{lemma}[theorem]{Lemma}

\def\ker{\mathop{\mathrm{Ker}}\nolimits}
\def\im{\mathop{\mathrm{Im}}\nolimits}
\def\dom{\mathop{\mathrm{Dom}}\nolimits}

\def\id{\mathop{\mathrm{id}}\nolimits}
\def\auto{\mathop{\mathrm{Aut}}\nolimits}
\def\endo{\mathop{\mathrm{End}}\nolimits}
\def\cen{\mathop{\mathrm{C}}\nolimits}
\def\N{\mathbb N}

\def\O{\mathcal{O}}
\def\PO{\mathcal{PO}}
\def\POI{\mathcal{POI}}
\def\PODI{\mathcal{PODI}}
\def\POD{\mathcal{POD}}
\def\OD{\mathcal{OD}}
\def\PT{\mathcal{PT}}
\def\T{\mathcal{T}}
\def\I{\mathcal{I}}
\def\S{\mathcal{S}}

\newcommand{\lastpage}{\addresss}

\newcommand{\addresss}{\small \sf  

\noindent{\sc De Biao Li}, 
School of Mathematics and Statistics, 
Lanzhou University, 
Lanzhou, 
730000 Gansu, 
P. R. China;\\
e-mail: lidb19@lzu.edu.cn

\medskip

\noindent{\sc V\'\i tor H. Fernandes}, 
Center for Mathematics and Applications (CMA), 
FCT NOVA and Department of Mathematics, FCT NOVA, 
Faculdade de Ci\^encias e Tecnologia, 
Universidade Nova de Lisboa, 
Monte da Caparica, 
2829-516 Caparica, 
Portugal; 
e-mail: vhf@fct.unl.pt. 
}

\author{De Biao Li and V\'\i tor H. Fernandes\footnote{This work is funded by national funds through the FCT - Funda\c c\~ao para a Ci\^encia e a Tecnologia, 
I.P., under the scope of the projects UIDB/00297/2020 and UIDP/00297/2020 (Center for Mathematics and Applications).}
}

\title{Endomorphisms of semigroups of monotone transformations}

\begin{document}

\maketitle

\vspace*{-1cm}

\begin{abstract}
In this paper, 
we characterize the monoid of endomorphisms of the semigroup of all monotone full transformations of a finite chain, as well as the monoids of endomorphisms of the semigroup of all monotone partial transformations and of the semigroup of all monotone partial permutations of a finite chain. 
\end{abstract}

\medskip



\section*{Introduction}

For $n\in\N$, let $\Omega_n$ be a finite set with $n$ elements.
Denote by $\PT_n$ 
the monoid (under composition) of all partial transformations of $\Omega_n$.
The submonoid of $\PT_n$ of all full transformations of $\Omega_n$ and the (inverse)
submonoid of all partial permutations (i.e. partial injective transformations) of $\Omega_n$ are denoted by $\T_n$ and $\I_n$,
respectively.
Also, denote by $\S_n$ the symmetric group on $\Omega_n$, i.e. the subgroup of $\PT_n$ of all permutations of $\Omega_n$. 

Now, suppose that $\Omega_n$ is a finite
chain with $n$ elements, e.g. $\Omega_n=\{1<2<\cdots <n\}$. 
We say that a transformation $s$ in $\PT_n$ is \textit{order-preserving} [\textit{order-reversing}] if $x\leqslant y$ implies
$xs\leqslant ys$ [$xs\geqslant ys$], for all $x,y \in \dom(s)$. 
A transformation that is either order-preserving or order-reversing is also called \textit{monotone}.
Observe that the
product of two order-preserving transformations or of two
order-reversing transformations is order-preserving and the
product of an order-preserving transformation by an
order-reversing transformation is order-reversing. 
Moreover, the product of two monotone transformations is monotone. 

Denote by $\PO_n$ the submonoid of $\PT_n$ of all
partial order-preserving transformations of $\Omega_n$.
As usual, $\O_n$ denotes the monoid $\PO_n\cap\T_n$ of all
full transformations of $\Omega_n$ that preserve the order. This monoid has
been largely studied, namely in \cite{Aizenstat:1962,Gomes&Howie:1992,Howie:1971,Lavers&Solomon:1999}.
The injective counterpart of $\O_n$ is the inverse monoid $\POI_n=\PO_n\cap\I_n$,
which is considered, for example,
in \cite{Cowan&Reilly:1995,Fernandes:1997,Fernandes:1998,Fernandes:2001,Fernandes:2002, Fernandes:2002survey}.

Wider classes of monoids are obtained when we take monotone transformations. 
In this way, we get
$\POD_n$, the submonoid of $\PT_n$ of all partial monotone 
transformations. Naturally, we
may also consider $\OD_n=\POD_n\cap\T_n$ and
$\PODI_n=\POD_n\cap\I_n$, the monoids of all monotone
full transformations and of all monotone partial permutations, respectively. 
These monoids were studied, for instance, in 
\cite{Araujo&etal:2011, 
Delgado&Fernandes:2004, 
Dimitrova&Koppitz:2008,
Dimitrova&Koppitz:2009, 
Fernandes:2008, 
Fernandes&Gomes&Jesus:2004,
Fernandes&Gomes&Jesus:2005,
Fernandes&Gomes&Jesus:2005b,
Fernandes&Gomes&Jesus:2011, 
Fernandes&Quinteiro:2012,
Fernandes&Quinteiro:2016, 
Gyudzhenov&Dimitrova:2006}. 

The following diagram, with respect to the inclusion relation and
where {\bf1} denotes the trivial monoid, clarifies the
relationship between these various semigroups:
%\vspace{1em}
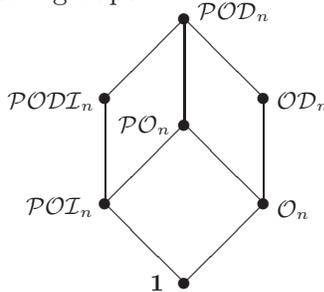
\begin{center}
\begin{picture}(94,100)(120,135)
\put(157,137){$\bullet$} \put(160,140){\line(1,1){30}}
\put(160,140){\line(-1,1){30}} \put(147,137){\footnotesize {\bf1}}
%\put(160,40){\line(-1,2){20}}
%
\put(127,167){$\bullet$} \put(130,170){\line(1,1){30}}
\put(130,170){\line(0,1){40}} \put(100,166){\footnotesize
$\POI_n$}
%\put(100,100){\line(-1,2){20}}
%
\put(187,167){$\bullet$} \put(190,170){\line(-1,1){30}}
\put(190,170){\line(0,1){40}} \put(195,165){\footnotesize
$\O_n$}
%\put(220,100){\line(-1,2){20}}
%
\put(127,207){$\bullet$} \put(130,210){\line(1,1){30}}
\put(93,206){\footnotesize $\PODI_n$}
%\put(60,140){\line(-1,2){20}}
%
\put(187,207){$\bullet$} \put(190,210){\line(-1,1){30}}
\put(195,206){\footnotesize $\OD_n$}
%\put(260,140){\line(-1,2){20}}
%
\put(157,197){$\bullet$} \put(160,200){\line(0,1){40}}
\put(135,197){\footnotesize $\PO_n$}
%\put(160,160){\line(-1,2){20}}
%
\put(157,237){$\bullet$} \put(165,240){\footnotesize
$\POD_n$}
%\put(160,240){\line(-1,2){20}}

\end{picture}
\end{center}
%\vspace{1em}

\medskip

Describing automorphisms and endomorphisms of transformation semigroups is a classical problem. 
For instance, they have been determined by Schein and Teclezghi \cite{Schein&Teclezghi:1997,Schein&Teclezghi:1998} 
for $\I_n$ in 1997 and for $\T_n$ in 1998,  
and for the Brauer-type semigroups by Mazorchuk \cite{Mazorchuk:2002} in 2002. 
Regarding semigroups of order-preserving transformations, 
in 1962 A\u{\i}zen\v{s}tat  \cite{Aizenstat:1962} gave a presentation for $\O_n$ from which it can be deduced that $\O_n$ has only one non-trivial automorphism, for $n\geqslant2$. 
More recently, 
in 2010 Fernandes et al. \cite{Fernandes&al:2010} found a description of the endomorphisms of $\O_n$ and 
in 2019 Fernandes and Santos \cite{Fernandes&al:2019} determined the endomorphisms of $\POI_n$ and of $\PO_n$. 
Descriptions of automorphisms of semigroups of order-preserving transformations and of some of their extensions, such as semigroups of monotone transformations or semigroups of orientation-preserving/reversing transformations, can be found in \cite{Araujo&etal:2011}. 

\smallskip 

In this paper, we give descriptions of the monoids of endomorphisms of the remain semigroups of the above diagram, 
namely of $\OD_n$, $\PODI_n$ and $\POD_n$, for $n\geqslant2$. 
Moreover, we also determine the number of endomorphisms of each of these semigroups.

\section{Preliminaries}\label{prelim} 

Let $S$ be a semigroup. For completion, we recall the definition of
the Green equivalence relations $\mathscr{R}$, $\mathscr{L}$, $\mathscr{H}$ and
$\mathscr{J}$: for all $u, v\in S$,
\begin{description}
\item $u\mathscr{R} v$ if and only if  $uS^1=vS^1$, 
\item $u\mathscr{L} v$ if and only if $S^1u=S^1v$, 
\item $u\mathscr{H} v$ if and only if $u\mathscr{L} v$ and $u\mathscr{R} v$, 
\item $u\mathscr{J} v$ if and only if $S^1uS^1=S^1vS^1$
\end{description}
(as usual, $S^1$ denotes $S$ with identity adjoined \textit{if necessary}). 
Associated to the Green relation $\mathscr{J}$ there is a quasi-order
$\leqslant_\mathscr{J}$ on $S$ defined by
$$
u\leqslant_\mathscr{J} v \text{ if and only if }  S^1uS^1\subseteq S^1vS^1,
$$
for all $u, v\in S$. Notice that, for every $u, v\in S$, we have
$u\,\mathscr{J}\,v$ if and only if $u\leqslant_\mathscr{J}v$ and
$v\leqslant_\mathscr{J}u$. Denote by $J_{u}$ the $\mathscr{J}$-class
of the element $u\in S$. As usual, a partial order relation
$\leqslant_\mathscr{J}$ is defined on the set $S/\mathscr {J}$ by
setting $J_{u}\leqslant_\mathscr{J}J_{v}$ if and only if
$u\leqslant_\mathscr{J}v$, for all $u, v\in S$. For $u, v\in S$, we
write $u<_\mathscr{J}v$ and also $J_{u}<_\mathscr{J}J_{v}$ if
and only if $u\leqslant_\mathscr{J}v$ and $(u, v)\not\in\mathscr{J}$. 
Recall that any endomorphism of semigroups preserves Green relations and the quasi-order $\leqslant_\mathscr{J}$. 

\smallskip 

Given a semigroup $S$, we denote by $E(S)$ the set of its idempotents.
An \textit{ideal} of $S$ is a subset $I$ of $S$ such that
$S^1IS^1\subseteq I$. By convenience, we admit the empty set as an
ideal. A \textit{Rees congruence} of $S$ is a congruence associated
to an ideal of $S$: if $I$ is an ideal of $S$, the Rees congruence
$\rho_I$ is defined by $(s,t)\in\rho_I$ if and only if $s=t$ or
$s,t\in I$, for all $s, t\in S$. 
The set of congruences of $S$, the group of automorphisms of $S$ and the monoid of endomorphisms of $S$ are denoted by
$\mathrm{Con}(S)$, $\auto(S)$ and $\endo(S)$, respectively. 

\medskip 

Let $S\in\{\O_n,\POI_n,\PO_n,\OD_n,\PODI_n,\POD_n\}$. 
We have the following descriptions of the Green relations in the semigroup $S$: 
\begin{description}
\item $s\mathscr{L} t$ if and only if $\im(s)=\im(t)$, 
\item $s\mathscr{R} t$ if and only if $\ker(s)=\ker(t)$, 
\item $s\mathscr{J} t$ if and only if $|\im(s)|=|\im(t)|$, and 
\item $s\mathscr{H} t$ if and only if $\ker(s)=\ker(t)$ and  $\im(s)=\im(t)$,
\end{description} 
for all $s,t\in S$. If $S=\POD_n$ or $\PODI_n$, for the Green relation $\mathscr{R}$, we have, even more simply, 
\begin{description}
\item
$s\mathscr{R} t$ if and only if $\dom(s)=\dom(t)$, 
\end{description} 
for all $s,t\in S$.

\smallskip 

Consider the following order-reversing full transformation
$$
\tau=\left(
\begin{array}{ccccc}
1&2&\cdots&n-1&n\\
n&n-1&\cdots&2&1
\end{array}
\right) 
$$
(a permutation of order two). 

Recall that
$\OD_n=\langle \O_n, \tau\rangle$, $\PODI_n=\langle
\POI_n, \tau \rangle$ and $\POD_n=\langle \PO_n, \tau \rangle$.

Next, let $T\in\{\O_n, \POI_n, \PO_n\}$ and 
$M=\langle T, \tau\rangle$. Then both $T$ and $M$ are regular monoids (moreover, if
$T=\POI_n$ then $T$ and $M$ are inverse monoids) and $E(M)=E(T)$.

Remember also that, for the partial order $\leqslant_\mathscr{J}$, the
quotients $M/\mathscr{J}$ and $T/\mathscr{J}$ are chains (with $n+1$ elements
for $T=\POI_n$ and $T=\PO_n$ and with $n$ elements for $T=\O_n$).
More precisely, if $S\in\{T,M\}$, then
$$S/\mathscr{J}=\{J_0^S <_\mathscr{J} J_1^S <_\mathscr{J} \cdots
<_\mathscr{J} J_n^S \}$$ when $T\in\{\POI_n, \PO_n\}$; and
$$
S/\mathscr{J}=\{J_1^S <_\mathscr{J} \cdots <_\mathscr{J} J_n^S\}
$$ when $T=\O_n$. Here 
$$
J_k^S=\{s\in S\mid |\im(s)|=k\},
$$ with
$k$ suitably defined.

For $S\in\{T,M\}$ and $0\leqslant  k\leqslant  n$, let 
$I_k^S=\{s\in S\mid |\im(s)|\leqslant k\}$. 
Clearly $I_k^S$ is an ideal of $S$. Since $S/\mathscr{J}$ is a
chain, it follows that
$$
\{I_k^S\mid 0\leqslant  k\leqslant  n\}
$$
is the set of all ideals of $S$. 
See \cite{Gomes&Howie:1992,Fernandes&Gomes&Jesus:2005,Fernandes:2001}. 

\smallskip

Observe that $T$ is an aperiodic monoid and that the
$\mathscr{H}$-classes of $M$ contained in $J_k^M$ have precisely
two elements (one of them belonging to $T$ and the other belonging
to $M\setminus T$) when $k\geqslant 2$. If $k=1$ then such
$\mathscr{H}$-classes are trivial.

In \cite{Aizenstat:1962b},  A\v\i zen\v stat proved that the congruences of $\O_n$ are exactly the identity and its $n$ Rees congruences. 
See \cite{Lavers&Solomon:1999} for another proof. 
Also, the congruences of $\POI_n$ and $\PO_n$ are exactly their $n+1$ Rees congruences. 
This has been shown, for $\POI_n$, by Derech \cite{Derech:1991} and, independently, by Fernandes \cite{Fernandes:2001} 
and, for $\PO_n$, by Fernandes et al. \cite{Fernandes&Gomes&Jesus:2005}. In short, 
$$
\{\rho_{I_k^T}\mid 0\leqslant  k\leqslant n\}
$$
is the set $\mathrm{Con}(T)$ of all
congruences of $T$. 

\smallskip 

Concerning the monoid $M\in\{\OD_n, \PODI_n, \POD_n\}$, 
for $1\leqslant  k\leqslant  n$, we can define a congruence
$\pi_k$ on $M$ by: for all $s,t\in M$, $s\:\pi_k \:t$ if and only
if
\begin{enumerate}
  \item
  $s=t$; or
  \item
  $s,t\in I_{k-1}^M$; or
  \item
  $s,t\in J_k^M$ and $s\:\mathcal{H}\:t$
\end{enumerate}
(see \cite[Lemma 4.2]{Fernandes:2000}).
For $0\leqslant  k\leqslant n$, denote by $\rho_k^M$ the Rees congruence
$\rho_{I_k^M}$ associated to the ideal $I_k^M$ of $M$ and by
$\omega$ the universal congruence of $M$. Clearly, for $n\geqslant 2$, we
have
$$1=\pi_1 \subsetneq \rho_1^M \subsetneq \pi_2 \subsetneq
\rho_2^M \subsetneq \cdots \subsetneq \pi_n \subsetneq
\rho_n^M=\omega.$$
Furthermore, Fernandes et al. \cite{Fernandes&Gomes&Jesus:2005} proved that, for $n\geqslant 2$, 
these are precisely all congruences of $M$:
$$
\mathrm{Con}(M)=\{1=\pi_1,\rho_1^M,\pi_2,\rho_2^M,\ldots,
\pi_n,\rho_n^M=\omega\}. 
$$

\medskip 

Let $S$ be a semigroup and let $u$ be a unit of any monoid containing $S$ such that $u^{-1}Su\subseteq S$. 
Then, it is easy to check that the mapping 
$\phi^u:S\longrightarrow S$ defined by $s\phi^u=u^{-1}su$, for all $s\in S$, is an automorphism of $S$, 
which is called an \textit{inner automorphism} of $S$.
If necessary, in case of ambiguity, we represent $\phi^u$ by $\phi^u_S$.

\smallskip 

Let us consider again the permutation $\tau$ of $\Omega_n$ (recall that $\tau\in\POD_n$) defined above. 
Then, for every $S\in\{\O_n,\POI_n,\PO_n,\OD_n,\PODI_n,\POD_n\}$, it is easy to verify that $\tau^{-1}S\tau=S$. 
Therefore, for $n\geqslant 2$, the permutation $\tau$ induces a non-trivial automorphism $\phi^\tau$ of $S$. 

\medskip

Now, let $S$ be a finite monoid and let $G$ be its group of units. 
Let $e$ and $f$ be two idempotents of $S$ such that $ef=fe=f$. 
Then, clearly, the mapping $\phi:S\longrightarrow S$ defined by $G\phi=\{e\}$ 
and $(S\setminus G)\phi=\{f\}$ is an endomorphism (of semigroups) of $S$. 
More generally, let $f$ be an idempotent of $S$, $H$ a subgroup of $S$ such that $fH=\{f\}=Hf$ and 
$\varphi: G\longrightarrow H$ a homomorphism.  
Then, it is also clear that the mapping $\phi:S\longrightarrow S$ defined by $\phi|_G=\varphi$ and $(S\setminus G)\phi=\{f\}$ is an endomorphism (of semigroups) of $S$. 

\smallskip 

On the other hand, observe that, given a semigroup $S$, an endomorphism $\phi$ of $S$ and an idempotent-generated subsemigroup $V$ of $S$ such that $E(V)=E(S)$, as idempotents apply to idempotents, we have $V\phi\subseteq V$ and so
we may consider the restriction $\phi|_V$ of $\phi$ to $V$ as an endomorphism of $V$. 

\smallskip 

Recall that the \textit{rank} of a transformation $s\in\PT_n$ is the size of its image, i.e. $|\im(s)|$. 

\medskip 

Next, we are going to construct a certain family of endomorphisms for each of the semigroups $\POI_n$, $\PO_n$, $\PODI_n$ and $\POD_n$. 

\smallskip 

For $k\in\{2,3,\ldots,n\}$, consider the following two transformations of $\O_n$ with rank $n-1$: 
$$
f_k=\left(\begin{array}{ccccccc}
1&\cdots&k-1&k&k+1&\cdots&n\\
1&\cdots&k-1&k-1&k+1&\cdots&n
\end{array}\right)
\quad\text{and}\quad 
g_k=\left(\begin{array}{ccccccc}
1&\cdots&k-2&k-1&k&\cdots&n\\
1&\cdots&k-2&k&k&\cdots&n
\end{array}\right).  
$$
It is easy to check that each of the $n-1$ $\mathscr{R}$-classes of transformations of rank $n-1$ 
of $\O_n$ or $\OD_n$ (recall that $E(\OD_n)=E(\O_n)$) 
has exactly two idempotents, 
namely $f_k$ and $g_k$, for some $k\in\{2,3,\ldots,n\}$.  See \cite[Lemma 2.4]{Fernandes&al:2019}. 
On the other hand, for $i\in\{1,2,\ldots,n\}$, let
$$
e_i=\left(\begin{array}{cccccc}
1&\cdots&i-1&i+1&\cdots&n\\
1&\cdots&i-1&i+1&\cdots&n
\end{array}\right), 
$$
which is an idempotent of $\POI_n$ of rank $n-1$. 
It is clear that each of the $n$ $\mathscr{R}$-classes of transformations of rank $n-1$ of $\POI_n$ or $\PODI_n$ (recall that $E(\PODI_n)=E(\POI_n)$) has exactly one  idempotent, namely $e_i$, for some $i\in\{1,2,\ldots,n\}$. 
Therefore, since $J_{n-1}^{\POD_n}=J_{n-1}^{\OD_n}\cup J_{n-1}^{\PODI_n}$ (a disjoint union), by the above observations, 
$J_{n-1}^{\POD_n}$ contains $n$ $\mathscr{R}$-classes exactly with one idempotent and $n-1$ $\mathscr{R}$-classes exactly with two idempotents. 

\smallskip 

Now, let us define a mapping $\phi_1:\POD_n\longrightarrow\POD_n$ by: 
\begin{enumerate}
\item $1\phi_1=1$, $\tau\phi_1=\tau$; 
\item For $s\in J_{n-1}^{\PODI_n}$, let $s\phi_1=\binom{i}{j}$, where $i,j\in\{1,2,\ldots,n\}$ are the unique indices such that $e_i\mathscr{R} s\mathscr{L} e_j$; 
\item For $s\in J_{n-1}^{\OD_n}$, let $s\phi_1=\left(\begin{array}{cc}
k-1&k\\
k_s&k_s
\end{array}\right)$, where $\{k_s\}=X_n\setminus\im(s)$ and $k\in\{2,3,\ldots,n\}$ is the unique index such that $s\mathscr{R} f_k$ (and $s\mathscr{R} g_k$); 
\item $I_{n-2}^{\POD_n}\phi_1=\{\emptyset\}$.
\end{enumerate}

Clearly, $\phi_1$ is well defined mapping. Moreover, it is a routine matter to show that $\phi$ is an endomorphism of $\POD_n$ which admits $\pi_{n-1}$ as kernel. Furthermore, it is clear that, for every $S\in\{\POI_n,\PO_n,\PODI_n\}$, we have $S\phi_1\subseteq S$ and so the restriction $\phi_1|_S$ of $\phi_1$ to $S$ may also be seen as an endomorphism of $S$.  

\medskip 

Next, recall that a subsemigroup $S$ of $\PT_n$ is said to be $\S_n$-\textit{normal} if $\sigma^{-1}S\sigma\subseteq S$, for all $\sigma\in\S_n$. 
In 1975, Sullivan \cite[Theorem 2]{Sullivan:1975} proved that $\auto(S)\simeq\S_n$, for any $\S_n$-normal subsemigroup $S$ of $\PT_n$ 
containing a constant mapping. 
Moreover, we obtain an isomorphism $\Phi:\S_n\longrightarrow\auto(S)$ by defining $\sigma\Phi=\phi^\sigma_S$,  
where $\phi^\sigma_S$ is the inner automorphism of $S$ associated to $\sigma$, for all $\sigma\in \S_n$. 

\medskip 

Let  $S\in\{\POI_n,\PO_n,\PODI_n,\POD_n\}$. 

Let $I_1^1=I_{1}^{S}\cup\{1\}$. It is clear that $I_1^1$ is an $\S_n$-normal subsemigroup of $\PT_n$ containing a constant mapping. 
Therefore, by Sullivan's Theorem, we have 
$$
\auto(I_1^1)=\{\phi^\sigma_{I_1^1}\mid \sigma\in\S_n\}\simeq\S_n. 
$$

Now, let $I_1^\tau=I_{1}^{S}\cup\{1,\tau\}$. Then $I_1^\tau$ is a subsemigroup of $\PT_n$ admitting $\{1,\tau\}$ as group of units. 
Therefore, for any automorphism $\phi$ of $I_1^\tau$, we must have $\tau\phi=\tau$, whence $I_1^1\phi=I_1^1$ and so 
the restriction $\phi|_{I_1^1}$ of $\phi$ to $I_1^1$, as a mapping from $I_1^1$ to $I_1^1$, is an automorphism of $I_1^1$. 
Hence $\phi|_{I_1^1}=\phi^\sigma_{I_1^1}$, for some $\sigma\in\S_n$. Let $i,j\in\{1,2,\ldots,n\}$. Then 
\begin{equation*}\label{centralizer}
\textstyle\binom{(n-i+1)\sigma}{j\sigma}=\sigma^{-1}\tau\binom{i}{j}\sigma=\left(\tau\binom{i}{j}\right)\phi=\tau\phi\binom{i}{j}\phi 
=\tau\sigma^{-1}\binom{i}{j}\sigma=\binom{n-i\sigma+1}{j\sigma},
\end{equation*}
whence $(n-i+1)\sigma=n-i\sigma+1$, i.e. $i\tau\sigma=i\sigma\tau$. So 
$$
\sigma\in\cen_{\S_n}(\tau)=\{\xi\in\S_n\mid \xi\tau=\tau\xi\},
$$
the centralizer of $\tau$ in $\S_n$. 
Notice that, $\sigma^{-1}\tau\sigma=\tau$ (and so $\sigma^{-1}I_1^\tau\sigma=I_1^\tau$), whence $\phi$ is the inner automorphism of $I_1^\tau$ associated to $\sigma$. 
On the other hand, given $\sigma\in\cen_{\S_n}(\tau)$, we clearly have $\sigma^{-1}I_1^\tau\sigma=I_1^\tau$ and so we obtain an inner automorphism $\phi_{I_1^\tau}^\sigma$ of $I_1^\tau$. Moreover, it is easy to conclude now that the mapping 
$\Phi:\cen_{\S_n}(\tau)\longrightarrow\auto(I_1^\tau)$ defined by $\sigma\Phi=\phi^\sigma_{I_1^\tau}$, 
for all $\sigma\in\cen_{\S_n}(\tau)$, 
is an isomorphism. Since $\tau$ has $\lfloor\frac{n}{2}\rfloor$ non-trivial cycles each of which with length $2$, we have 
$$
\textstyle|\auto(I_1^\tau)|=|\cen_{\S_n}(\tau)|=\lfloor\frac{n}{2}\rfloor\hskip-.1em!\, 2^{\lfloor\frac{n}{2}\rfloor}
$$
(see \cite[Proposition 23 (page 133)]{Ledermann:1976}). 

For  $S\in\{\POI_n,\PO_n\}$, let $\phi_\sigma^S=\phi_1|_S\phi^\sigma_{I_1^1}$, considered as a mapping from $S$ to $S$, 
for all $\sigma\in\S_n$. Clearly,  
$$
\{\phi_\sigma^S\mid \sigma\in\S_n\}
$$
is a set of $n!$ distinct endomorphisms of $S$. 
On the other hand, for  $S\in\{\PODI_n,\POD_n\}$, 
let $\phi_\sigma^S=\phi_1|_S\phi^\sigma_{I_1^\tau}$, considered as a mapping from $S$ to $S$, 
for all $\sigma\in\cen_{\S_n}(\tau)$. Clearly,  
$$
\{\phi_\sigma^S\mid \sigma\in\cen_{\S_n}(\tau)\}
$$
is a set of $\lfloor\frac{n}{2}\rfloor\hskip-.1em!\, 2^{\lfloor\frac{n}{2}\rfloor}$ distinct endomorphisms of $S$.

\medskip 

For general background on semigroups and standard notations, we refer the reader to Howie's
book \cite{Howie:1995}. 

\medskip 

From now on, we consider $n\geqslant2$, whenever not explicitly mentioned.

\section{Endomorphisms of $\OD_n$, $\PODI_n$ and $\POD_n$}\label{main} 

We begin this section by recalling the following results by Fernandes et al. \cite{Fernandes&al:2010} 
and Fernandes and Santos \cite{Fernandes&al:2019} above mentioned. 

\begin{theorem}[{\cite[Theorem 1.1]{Fernandes&al:2010} and \cite[Theorem 3.3]{Fernandes&al:2019}}] \label{endon}
For $n\geqslant2$, let $T\in\{\O_n,\POI_n,\PO_n\}$ and $\phi: T\longrightarrow T$ be any mapping. Then $\phi$ is an
endomorphism of the semigroup $T$ if and only if one of the following properties holds:
\begin{enumerate}

\item $\phi$ is an automorphism and so $\phi$ is the identity or $\phi=\phi^\tau$;

\item if $T\in\{\POI_n,\PO_n\}$ and $\phi=\phi_\sigma^T$, for some $\sigma\in\S_n$; 

\item there exist idempotents $e,f\in T$ with $e\neq f$ and $ef=fe=f$ such
that $1\phi=e$ and $(T\setminus\{1\})\phi=\{f\}$;

\item $\phi$ is a constant mapping with idempotent value. 

\end{enumerate}
\end{theorem}

And, their corollaries: 

\begin{corollary}[{\cite[Theorem 1.2]{Fernandes&al:2010} and \cite[Theorems 3.4 and 3.7]{Fernandes&al:2019}}]\label{sendon} 
Let $n\geqslant2$. Then: 
\begin{enumerate}

\item the semigroup $\O_n$ has 
$2+\sum_{i=0}^{n-1}\binom{n+i}{2i+1}F_{2i+2}$ endomorphisms, 
where $F_{2i+2}$ denotes the $(2i+2)$th Fibonacci number; 

\item the semigroup $\POI_n$ has $2+n!+3^n$ endomorphisms; 

\item the semigroup $\PO_n$ has 
$$\textstyle
3 + n! + (\sqrt{5})^{n-1}\left(\left(\frac{\sqrt{5}+1}{2}\right)^n -\left(\frac{\sqrt{5}-1}{2}\right)^n\right) + 
\sum_{k=1}^{n} (\sqrt{5})^{k-1}\left(\left(\frac{\sqrt{5}+1}{2}\right)^k 
-\left(\frac{\sqrt{5}-1}{2}\right)^k\right)\sum_{i=k}^{n}\binom{n}{i}\binom{i+k-1}{2k-1}
$$ 
endomorphisms. 
\end{enumerate}
\end{corollary}

\smallskip 

Now, let $T\in\{\O_n, \PO_n\}$ and take $M=\langle T, \tau\rangle$. 
Since $T$ is an idempotent-generated semigroup (see \cite{Aizenstat:1962,Gomes&Howie:1992})
and $E(T)=E(M)$, then a restriction to $T$ of any endomorphism of $M$ may also be considered as an endomorphism of $T$. 
On the other hand, Lemma \ref{poirest} below establishes that this last statement is also true for $T=\POI_n$ (and $M=\PODI_n$).    

\smallskip 

Let us consider the following elements of $\POI_n$: 
$$
x_0=\left(
\begin{array}{cccc}
2&\cdots&n-1&n\\
1&\cdots&n-2&n-1
\end{array}
\right) 
\quad\text{and}\quad 
x_i=\left(
\begin{array}{ccc|c|ccc}
1&\cdots&n-i-1&n-i&n-i+2&\cdots&n\\
1&\cdots&n-i-1&n-i+1&n-i+2&\cdots&n
\end{array}
\right), 
$$
for $i\in\{1,2,\ldots,n-1\}$. 
Then $\{1,x_0,x_1,\ldots,x_{n-1}\}$ is a (semigroup) generating set of $\POI_n$. 
Moreover, these transformations satisfy the following equalities: 
\begin{enumerate}
\item $x_ix_0=x_0x_{i+1}$, for $1\leqslant i\leqslant n-2$;
\item  $x_{i+1}x_ix_{i+1}=x_{i+1}x_i$, 
for $1\leqslant i\leqslant n-2$;
\item $x_0x_1\cdots x_{n-1}x_0=x_0$. 
\end{enumerate}
See \cite{Fernandes:2001}.  
Observe that $x_{i+1}x_i,x_0x_{i+1}\in J_{n-2}^{\POI_n}$ for all $1\leqslant i\leqslant n-2$. 

\begin{lemma}\label{poirest} 
Let $n\geqslant2$ and $\phi$ be an endomorphism of $\PODI_n$. 
Then $\POI_n\phi\subseteq\POI_n$.  
\end{lemma}
\begin{proof}
First, notice that $1\phi$ is an idempotent and so $1\phi\in\POI_n$. 

\smallskip 

Let $n=2$. Then $\tau$ is the only non-order-preserving element of $\PODI_2$. Suppose that $\tau=s\phi$, for some $s\in\POI_2$. 
Hence $s=\binom{1}{2}$ or $s=\binom{2}{1}$ (these are the only non-idempotents of $\POI_2$) 
and so $\emptyset\phi=(s^2)\phi=(s\phi)^2=\tau^2=1$ and 
$\emptyset\phi=(s^3)\phi=(s\phi)^3=\tau^3=\tau$, which is a contradiction. Thus $\POI_2\phi\subseteq\POI_2$. 

\smallskip 

Next, let $n=3$. 
In order to obtain a contradiction, suppose there exists $i\in\{0,1,2\}$ such that $x_i\phi\not\in\POI_3$. 
Then $J_2^{\PODI_3}\phi\subseteq J_m^{\PODI_3}$, for some $m\geqslant2$. 

If $\pi_2\subsetneq\ker(\phi)$ then $\rho_2^{\PODI_3}\subseteq\ker(\phi)$, whence $|I_2^{\PODI_3}\phi|=1$ and so 
$x_i\phi$ is an idempotent, which is a contradiction. Hence $\ker(\phi)\subseteq\pi_2$ and so $\phi$ is injective in 
$J_2^{\POI_3}\cup\{1,\tau\}$. 
Therefore $m=2$, $\tau\phi=\tau$ and $I_1^{\PODI_3}\phi\subseteq I_1^{\PODI_3}$. 

Suppose $i=0$. 
Since $x_0x_1x_2x_0=x_0$, we have $x_0\phi x_1\phi x_2\phi x_0\phi=x_0\phi\not\in\POI_3$, whence 
$x_1\phi x_2\phi\not\in\POI_3$ and so $x_1\phi\in\POI_3$ or $x_2\phi\in\POI_3$. 
Since $\tau x_0 x_1\tau=x_1$ and $\tau x_2x_0\tau=x_2$, 
we have $\tau (x_0\phi)( x_1\phi)\tau=x_1\phi\in J_2^{\PODI_3}$ and $\tau (x_2\phi)(x_0\phi)\tau=x_2\phi\in J_2^{\PODI_3}$. 
If $x_1\phi\in\POI_3$ then $x_1\phi=\tau (x_0\phi)( x_1\phi)\tau\not\in\POI_3$, which is a contradiction. 
If $x_2\phi\in\POI_3$ then $x_2\phi=\tau (x_2\phi)( x_0\phi)\tau\not\in\POI_3$, which is again a contradiction. 
Thus $i\neq0$ and so $x_0\phi\in\POI_3$. 

Now, from $x_0\phi x_1\phi x_2\phi x_0\phi=x_0\phi\in J_2^{\POI_3}$, 
we may conclude that $x_1\phi x_2\phi\in\POI_3$ and so, 
since $x_1\phi\not\in\POI_3$ or $x_2\phi\not\in\POI_3$, we have 
 $x_1\phi, x_2\phi\not\in\POI_3$. 
 
Since $x_1^2$ and $x_2^2$ are idempotents with rank $1$, then $(x_1\phi)^2$ and $(x_2\phi)^2$ are also idempotents and their rank must be $1$ or $0$. 
On the other hand, we may routinely check that 
$$
\left\{x\in J_2^{\PODI_3}\setminus\POI_3\mid\mbox{$x^2$ is an idempotent of rank $1$ or $0$}\right\} = 
\left\{
\left(
\begin{array}{cc}
1&2\\
3&2
\end{array}
\right), 
\left(
\begin{array}{cc}
2&3\\
2&1
\end{array}
\right)
\right\},
$$
whence $\left\{x_1\phi,x_2\phi\right\}
=\left\{
\left(
\begin{array}{cc}
1&2\\
3&2
\end{array}
\right), 
\left(
\begin{array}{cc}
2&3\\
2&1
\end{array}
\right)
\right\}
$
and so $x_1\phi=(x_2\phi)^{-1}=x_2^{-1}\phi$. 
Since $\ker(\phi)\subseteq\pi_2$, it follow that $x_1\mathscr{H}x_2^{-1}$, which is a contradiction. 

Thus $x_i\phi\in\POI_3$ for all $i\in\{0,1,2\}$ and so $\POI_3\phi\subseteq\POI_3$. 

\smallskip 

Now, consider $n\geqslant4$. Let $\ell,m\in\{1,2,\ldots,n\}$ be such that 
$J_{n-2}^{\PODI_n}\phi\subseteq J_\ell^{\PODI_n}$ and $J_{n-1}^{\PODI_n}\phi\subseteq J_m^{\PODI_n}$. 
Then $\ell\leqslant m$. 

\smallskip

We begin by assuming that $\ell\geqslant 2$. 

In order to obtain a contradiction, suppose there exists $i\in\{0,1,\ldots, n-1\}$ such that $x_i\phi\not\in\POI_n$. 

Suppose $i=0$. 
Since $x_0x_1\cdots x_{n-1}x_0=x_0$, we have 
$x_0\phi (x_1\phi\cdots x_{n-1}\phi) x_0\phi=x_0\phi\not\in\POI_n$ and so there must be $j\in\{1,2,\ldots, n-1\}$ such 
$x_j\phi\not\in\POI_n$. 
If $j=n-1$ then $x_{n-2}x_0=x_0x_{n-1}$ implies $x_{n-2}\phi x_0\phi =x_0\phi x_{n-1}\phi\in J_{\ell}^{\PODI_n}$ and so 
$x_{n-2}\phi\not\in\POI_n$, since $\ell\geqslant2$ and $x_0\phi,x_{n-1}\phi\not\in\POI_n$. 
On the other hand, we also have $x_{n-1}x_{n-2}x_{n-1}=x_{n-1}x_{n-2}$, 
whence $x_{n-1}\phi x_{n-2}\phi x_{n-1}\phi=x_{n-1}\phi x_{n-2}\phi\in J_{\ell}^{\PODI_n}$, 
which is a contradiction, since $\ell\geqslant2$ and $x_{n-2}\phi, x_{n-1}\phi\not\in\POI_n$. 
Therefore $j\in\{1,2,\ldots, n-2\}$ and so we have $x_jx_0=x_0x_{j+1}$. 
Hence $x_j\phi x_0\phi=x_0\phi x_{j+1}\phi\in J_{\ell}^{\PODI_n}$. 
Since $\ell\geqslant2$ and $x_0\phi,x_j\phi\not\in\POI_n$, we deduce that also $x_{j+1}\phi\not\in\POI_n$. 
On the other hand, we also have $x_{j+1}x_jx_{j+1}=x_{j+1}x_j$, 
whence $x_{j+1}\phi x_j\phi x_{j+1}\phi=x_{j+1}\phi x_j\phi\in J_{\ell}^{\PODI_n}$, 
which is again a contradiction, since $\ell\geqslant2$ and $x_j\phi, x_{j+1}\phi\not\in\POI_n$. 
Thus $i\neq0$ and so $x_0\phi\in\POI_n$. 

Let $i\in\{1,2,\ldots, n-1\}$ be the smallest index such that $x_i\phi\not\in\POI_n$. 

If $i\geqslant 2$ then $x_{i-1}x_0=x_0x_i$ and so $x_{i-1}\phi x_0\phi=x_0\phi x_i\phi\in J_{\ell}^{\PODI_n}$, 
which is a contradiction, since $\ell\geqslant 2$, $x_0\phi,x_{i-1}\phi\in\POI_n$ and $x_i\phi\not\in\POI_n$. 
Hence $i=1$. From $x_1x_0=x_0x_2$, we get  $x_1\phi x_0\phi=x_0\phi x_2\phi\in J_{\ell}^{\PODI_n}$ and, 
as $x_0\phi\in\POI_n$ and $x_1\phi\not\in\POI_n$, we deduce that $x_2\phi\not\in\POI_n$. 
On the other hand, we also have $x_2x_1x_2=x_2x_1$, whence $x_2\phi x_1\phi x_2\phi=x_2\phi x_1\phi\in J_{\ell}^{\PODI_n}$, 
which is again a contradiction, since $\ell\geqslant 2$ and $x_1\phi,x_2\phi\not\in\POI_n$. 

Thus $x_i\phi\in\POI_n$ for all $i\in\{0,1,\ldots, n-1\}$ and so $\POI_n\phi\subseteq\POI_n$. 

\smallskip 

Now, let $\ell\leqslant1$. 

Suppose that $\ker(\phi)=\pi_k$, with $k\leqslant n-2$, or $\ker(\phi)=\rho_k^{\PODI_n}$, with $k\leqslant n-3$. 
Then $\phi$ is injective in $J_{n-2}^{\POI_n}$ and so 
$|E(J_{n-2}^{\POI_n})|\leqslant |E(J_\ell^{\POI_n})|\leqslant n$, 
which is a contradiction, since $|E(J_{n-2}^{\POI_n})|=\binom{n}{n-2}>n$, for $n\geqslant4$. 

Hence $\ker(\phi)\in\left\{\rho_{n-2}^{\PODI_n},\pi_{n-1},\rho_{n-1}^{\PODI_n},\pi_n,\rho_n^{\PODI_n}=\omega\right\}$.  
Clearly, if $\ker(\phi)\in\left\{\rho_{n-1}^{\PODI_n},\pi_n,\rho_n^{\PODI_n}\right\}$ then $\POI_n\phi\subseteq\POI_n$. 
Therefore, let us admit that $\ker(\phi)\in\left\{\rho_{n-2}^{\PODI_n},\pi_{n-1}\right\}$. 
Then $I_{n-2}^{\PODI_n}\phi=\{f\}$, for some idempotent $f\in J_\ell^{\PODI_n}$, 
and $m\geqslant1$ (recall that $J_{n-1}^{\PODI_n}\phi\subseteq J_m^{\PODI_n}$), since $\phi$ is injective in $J_{n-1}^{\POI_n}$. 

Let $e_1,e_2,\ldots, e_n$ be the $n$ distinct idempotents of $J_{n-1}^{\PODI_n}$ considered in Section \ref{prelim} 
(i.e. $e_i$ is  the partial identity with domain $\Omega_n\setminus\{i\}$, for $i\in\{1,2,\ldots, n\}$). 

Let $i\in\{1,2,\ldots, n\}$. Then $f=\emptyset\phi=(\emptyset e_i)\phi=\emptyset\phi e_i\phi=f(e_i\phi)$, whence 
$\dom(f)\subseteq\dom(e_i\phi)$. Since $(\emptyset,e_i)\not\in\ker(\phi)$ then $f=\emptyset\phi\neq e_i\phi$ and so 
$\dom(f)\subsetneq\dom(e_i\phi)$. 

On the other hand, for $1\leqslant i<j\leqslant n$, we have $e_ie_j\in I_{n-2}^{\PODI_n}$, whence 
$f=(e_ie_j)\phi=e_i\phi e_j\phi$ and so $\dom(f)=\dom(e_i\phi)\cap\dom(e_j\phi)$.  

Take $\ell_i\in\dom(e_i\phi)\setminus\dom(f)$, for $i\in\{1,2,\ldots, n\}$. 
If there exist $1\leqslant i<j\leqslant n$ such that $\ell_i=\ell_j$ then $\ell_i\in(\dom(e_i\phi)\cap\dom(e_j\phi))\setminus\dom(f)=\emptyset$, 
which is a contradiction. Hence, $\ell_i\neq\ell_j$ for all $1\leqslant i<j\leqslant n$ and so 
$\{\ell_1,\ell_2,\ldots,\ell_n\}=\{1,2,\ldots, n\}$. 
Now, since $\ell_i\not\in\dom(f)$ for all $i\in\{1,2,\ldots, n\}$, we conclude that $\dom(f)=\emptyset$, i.e. $f=\emptyset$,  
and $\dom(e_i\phi)\cap\dom(e_j\phi)=\emptyset$ for all $1\leqslant i<j\leqslant n$. 
Moreover, we have  $|\dom(e_i\phi)|=m\geqslant1$ for $i\in\{1,2,\ldots, n\}$. 
Hence $|\dom(e_i\phi)|=1$ for $i\in\{1,2,\ldots, n\}$ and so $m=1$. 
Thus $J_{n-1}^{\PODI_n}\phi\subseteq J_1^{\PODI_n}\subseteq\POI_n$ and so $\POI_n\phi\subseteq\POI_n$, as required.
\end{proof} 

\smallskip 

We may now deduce from Theorem \ref{endon} the following descriptions 
of the endomorphisms of $\OD_n$, $\PODI_n$ and $\POD_n$: 

\begin{theorem} \label{mainthOD}
For $n\geqslant2$, let $M\in\{\OD_n,\PODI_n,\POD_n\}$ and $\phi: M\longrightarrow M$ be any mapping. 
Then $\phi$ is an endomorphism of the semigroup $M$ if and only if one of the following properties holds:
\begin{enumerate}
\item  $\phi$ is the identity or $\phi=\phi^\tau$ and so $\phi$ is an automorphism;

\item if $M\in\{\PODI_n,\POD_n\}$ and $\phi=\phi_\sigma^M$, for some $\sigma\in\cen_{\S_n}(\tau)$; 

\item if $M\in\{\PODI_n,\POD_n\}$ and there exists a non-idempotent group element $h$ of $M$ 
such that $\tau\phi=h$, $1\phi=h^2$ and $(M\setminus\{1,\tau\})\phi=\{\emptyset\}$;

\item there exist a non-idempotent group element $h$ of $M$ and an idempotent transformation $f$ of $M$ 
with rank $1$ such that $hf=fh=f$, 
$\tau\phi=h$, $1\phi=h^2$ and $(M\setminus\{1,\tau\})\phi=\{f\}$;

\item there exist idempotents $e,f\in M$ with $e\neq f$ and $ef=fe=f$ such
that $\{1,\tau\}\phi=\{e\}$ and $(M\setminus\{1,\tau\})\phi=\{f\}$;

\item $\phi$ is a constant mapping with idempotent value.
\end{enumerate}
\end{theorem}
\begin{proof} Obviously, if Property 1, 2 or 6 holds then $\phi$ is an endomorphism of $M$. 
On the other hand, if Property 3, 4 or 5 holds then, as observed above in a more general context, 
$\phi$ is an endomorphism of $M$. In fact, 
this is immediate for Property 5. 
On the other hand, in the case of Property 3 or 4, $h^2\neq h$, $h^2$ is an idempotent and $h\mathscr{H}h^2$, 
whence $\phi|_{\{1,\tau\}}: \{1,\tau\}\longrightarrow\{h^2,h\}$ is an isomorphism of groups. 
Moreover, for Property 3, we have immediately $\emptyset\{h^2,h\}=\{\emptyset\}=\{h^2,h\}\emptyset$.   
Regarding Property 4, since 
$hf=fh=f$, we clearly have  
$f\{h^2,h\}=\{f\}=\{h^2,h\}f$. 
Therefore, in all cases, $\phi$ is an endomorphism of $M$. 

\smallskip 

Conversely, assume that $\phi: M\longrightarrow M$ is an endomorphism. 
Let $T\in\{\O_n, \POI_n, \PO_n\}$ be such that $M=\langle T, \tau\rangle$. 
Then, as mentioned above, $\phi|_T: T\longrightarrow T$ is an endomorphism of $T$. 

\smallskip 

We start by supposing that $\phi|_T$ is an automorphism of $T$. Then $1\phi=1$ and, 
since $1\mathscr{H}\tau$ implies $1\phi\mathscr{H}\tau\phi$, it follows that $\tau\phi=\tau$. 
In fact, if $\tau\phi=1$ then $(1,\tau)\in\ker(\phi)$ and so 
$\pi_n\subseteq\ker(\phi)$, which contradicts that $\phi|_T$ is an automorphism of $T$, whence $\tau\phi=\tau$.

Let $s$ be any element of $M\setminus T$. Then $s\tau\in T$. Hence, 
if $\phi|_T=\id_T$ then 
$$
s\phi=(s\tau\tau)\phi = (s\tau)\phi \tau\phi = (s\tau)\tau=s
$$
and 
if $\phi|_T=\phi^\tau_T$ then 
$$
s\phi=(s\tau\tau)\phi = (s\tau)\phi \tau\phi = (\tau(s\tau)\tau)\tau=\tau s\tau=(s)\phi^\tau_M. 
$$
Thus $\phi=\id_M$ or $\phi=\phi^\tau_M$. 

\smallskip 

Next, suppose that $T\in\{\POI_n,\PO_n\}$ and $\phi|_T=\phi_\sigma^T$, for some $\sigma\in\S_n$. 

Let $s\in J_{n-2}^T$. Then 
$s\phi=s\phi_\sigma^T=s\phi_1\phi_{I_1^1}^\sigma=
\emptyset\phi_{I_1^1}^\sigma=\sigma^{-1}\emptyset\sigma=\emptyset$. 
Hence 
$J_{n-2}^T\phi=\{\emptyset\}$. 
If $n=2$ then $I_{n-2}^M\phi=J_{n-2}^T\phi=\{\emptyset\}$. 
On the other hand, if $n\geqslant3$ then 
$\phi$ is not injective in $J_{n-2}^T$, whence $\ker(\phi)\not\subseteq\pi_{n-2}$ and so 
$\rho_{n-2}^M\subseteq\ker(\phi)$, from which follows again $I_{n-2}^M\phi=\{\emptyset\}$. 

Now, notice that $\phi$ is injective in $J_{n-1}^T$, whence $\ker(\phi)\subsetneq\rho_{n-1}^M$.  
If $n=2$ then it follows immediately that $\ker(\phi)=\pi_1=\pi_{n-1}$. So, take $n\geqslant3$ 
and let $s\in J_{n-1}^M\setminus J_{n-1}^T$ and $t\in J_{n-1}^T$ be such that $s\mathscr{H}t$. 
Then $s\phi\mathscr{H}t\phi$ and so $s\phi=t\phi$, since $t\phi=t\phi_1\phi_{I_1^1}^\sigma\in J_1^T=J_1^M$. 
Hence $\rho_{n-2}^M\subsetneq\ker(\phi)$ and so we deduce again that $\ker(\phi)=\pi_{n-1}$. 

Then, we have $\tau\phi\neq 1\phi$ and, since $1\phi=1\phi_1\phi_{I_1^1}^\sigma=\sigma^{-1}1\sigma=1=1\phi=(\tau^2)\phi=(\tau\phi)^2$, 
it follows that $\tau\phi=\tau$. 

Let $i,j\in\{1,2,\ldots,n\}$ and take $s\in J_{n-1}^{\POI_n}$ such that $e_i\mathscr{R} s\mathscr{L} e_j$. 
Then $s\phi_1=\binom{i}{j}$ and, since $e_{n-i+1}\mathscr{R} \tau s\mathscr{L} e_j$, we get $(\tau s)\phi_1=\binom{n-i+1}{j}$. 
Hence 
\begin{align*}\textstyle
\binom{(n-i+1)\sigma}{j\sigma}=\sigma^{-1}\binom{n-i+1}{j}\sigma=\binom{n-i+1}{j}\phi_{I_1^1}^\sigma=
(\tau s)\phi_1\phi_{I_1^1}^\sigma=(\tau s)\phi=\qquad\qquad\\ \textstyle
\tau\phi s\phi=\tau (s\phi_1\phi_{I_1^1}^\sigma)=
\tau \left(\binom{i}{j}\phi_{I_1^1}^\sigma\right)=
\tau\sigma^{-1}\binom{i}{j}\sigma=\binom{n-i\sigma+1}{j\sigma}
\end{align*}
and so $(n-i+1)\sigma=n-i\sigma+1$, i.e. $i\tau\sigma=i\sigma\tau$. Thus  $\sigma\in\cen_{\S_n}(\tau)$. 

Therefore, we conclude that $\phi=\phi_\sigma^M$, for some $\sigma\in\cen_{\S_n}(\tau)$. 

\smallskip 

Now, assume there exist idempotents $e,f\in T$ with $e\neq f$ and $ef=fe=f$ such
that $1\phi=e$ and $(T\setminus\{1\})\phi=\{f\}$. 

If $\tau\phi=e$ then, for any $s\in M\setminus(T\cup\{\tau\})$, we have $s=\tau(\tau s)$ and $\tau s\in T\setminus\{1\}$, 
whence $s\phi=\tau\phi(\tau s)\phi=ef=f$ and so $\phi$ is an endomorphism verifying Property 5. 

On the other hand, admit that $\tau\phi\not=e$. Let $h=\tau\phi$. Then $h\neq e$ and $h^2=(\tau\phi)^2=(\tau^2)\phi=1\phi=e$, 
whence $h$ is a non-idempotent group element of $M$ and so $h\not\in T$. 

Let $s$ be a constant transformation of $M$. 
Then $s,s\tau,\tau s\in T\setminus\{1\}$ and so $fh=s\phi\tau\phi=(s\tau)\phi=f=(\tau s)\phi=\tau\phi s\phi=hf$. 
Since $f\in T$ and $h\not\in T$, the equality $f=fh$ (or $f=hf$) allows us to conclude that 
 $f$ has rank $1$, if $T=\O_n$, and that $f$ has rank $1$ or $f$ is the empty transformation, 
if $T\in\{\POI_n,\PO_n\}$. 

Now, let $s\in M\setminus(T\cup\{\tau\})$. Then 
$s=(s\tau)\tau$ and $s\tau\in T\setminus\{1\}$, 
whence $s\phi=(s\tau)\phi\tau\phi=fh=f$. 

Thus $\phi$ is an endomorphism verifying Property 3 or 4. 

\smallskip 

Finally, we admit that $\phi|_T$ is a constant mapping with idempotent value. 
Let $s\in T\setminus\{1\}$. Then $(1,s)\not\in\pi_n$ and $(1,s)\in\ker(\phi)$. Hence $\ker(\phi)=\omega$, i.e. 
$\phi$ is also a constant mapping with idempotent value, as required. 
\end{proof}

As an immediate corollary, we have: 

\begin{corollary}
For $n\geqslant2$, let $M\in\{\OD_n,\PODI_n,\POD_n\}$. Then $\auto(M)=\{\id,\phi^{\tau}\}$.
\end{corollary}

Let $T\in\{\O_n, \POI_n,\PO_n\}$ and take $M=\langle T, \tau\rangle$.  
Notice that,
at the time when Theorem \ref{endon} was proved it was already known that $\auto(T)=\{\id,\phi^{\tau}\}$, for $n\geqslant2$. 
However, for the monoid $M$, it was known until now that $\auto(M)=\{\id,\phi^{\tau}\}$ but only for $n\geq10$. 
See \cite[Theorem 5.4]{Araujo&etal:2011}. 

\smallskip 

Now, we will count the number of endomorphisms of $\OD_n$, $\PODI_n$ and $\POD_n$. 

As above, let $T\in\{\O_n, \POI_n,\PO_n\}$ and $M=\langle T, \tau\rangle$.  

We begin by calculating the number of endomorphisms of $M$ satisfying Property 4 of Theorem \ref{mainthOD}.

Let $h$ be a non-idempotent group element of $M$.  
Then $|\im(h)|\geqslant2$ and $h$ is an order-reversing transformation such that $h^2$ is an idempotent and $h\mathscr{H}h^2$.  
Thus $h$ has fixed points if and only if $|\im(h)|$ is odd and, in this case, it has exactly one. 

On the other hand, we may also conclude that, 
for $2\leqslant i \leqslant n$, the number of non-idempotent group elements of $M$ belonging to $J^{M}_i$ is equal to 
$|E(J_i^T)|$. 

Let $F(h)=\{f\in E(M)\cap J_1^M\mid hf=fh=f\}$. 
Then the number of endomorphisms of $M$ satisfying Property 4 of Theorem \ref{mainthOD} such that $\tau\phi=h$ is $|F(h)|$. 
Notice that, $fh=f$ if and only if $h$ fixes the image of $f$. Then, if $|\im(h)|$ is even, clearly, we have $|F(h)|=0$. 
If $M\in\{\OD_n, \PODI_n\}$ then, for each $j\in\Omega_n$, there exists exactly one idempotent $f$ of $M$ such that $\im(f)=\{j\}$. 
Moreover, for $j\in\Omega_n$, if $jh=j$ and $f$ is the idempotent of $M$ such that $\im(f)=\{j\}$, then $hf=f$. Thus, if 
$|\im(h)|$ is odd and $M\in\{\OD_n, \PODI_n\}$,  it is also clear that $|F(h)|=1$.  
Now, let $M=\POD_n$ and suppose that  $|\im(h)|$ is odd. 
If $f\in F(h)$ then, from we equality $hf=f$, we deduce that $\dom(f)\subseteq\dom(h)$ 
and $x\in\dom(f)$ if and only if $xh\in\dom(f)$, for all $x\in\Omega_n$.  
Therefore, $jh^{-1}\subseteq\dom(f)$, where $j$ is the fixed point of $h$ (and the image of $f$). 
Moreover, for each pair 
$i_1,i_2\in\dom(h)\cap\im(h)$ such that $i_1\neq i_2$, $i_1h=i_2$ and $i_2h=i_1$,
we have $(i_1h^{-1}\cup i_2h^{-1})\cap\dom(f)=\emptyset$ or $(i_1h^{-1}\cup i_2h^{-1})\subseteq\dom(f)$. 
Thus, as we have $\frac{|\im(h)|-1}{2}$ of these pairs, it follows that $|F(h)|= 2^{\frac{|\im(h)|-1}{2}}$. 

Next, for $1\leqslant i\leqslant n$, recall that:
\begin{enumerate}
\item  $|E(J_i^{\O_n})|=\binom{n+i-1}{2i-1}$ (see \cite[Corollary 4.4]{Laradji&Umar:2006}); 

\item $|E(J_i^{\POI_n})|=\binom{n}{i}$, i.e. the number of partial identities of rank $i$;
 
\item $|E(J_i^{\PO_n})|=\sum_{k=i}^{n}\binom{n}{k}\binom{k+i-1}{2i-1}$
(see \cite[Lemma 3.6]{Fernandes&al:2019}).
\end{enumerate}

Therefore, being 
\begin{equation*}
\partial_i=
\begin{cases}
0 & \mbox{if $i$ is even}\\
1 &\mbox{if $i$ is odd and $M\in\{\OD_n, \PODI_n\}$}\\
2^{\frac{i-1}{2}} & \mbox{if $i$ is odd and $M=\POD_n$}, 
\end{cases}
\end{equation*} 
the number of endomorphisms of $M$ satisfying Property 4 of Theorem \ref{mainthOD} is 
\begin{enumerate}
\item $\sum_{i=2}^{n}\binom{n+i-1}{2i-1}\partial_i=\sum_{i=3}^{n}\binom{n+i-1}{2i-1}\partial_i = 
\sum_{i=1}^{\lfloor\frac{n-1}{2}\rfloor}\binom{n+2i}{4i+1}$, 
if $M=\OD_n$;

\item $\sum_{i=2}^{n}\binom{n}{i}\partial_i = \sum_{i=3}^{n}\binom{n}{i}\partial_i = 
\sum_{i=1}^{\lfloor\frac{n-1}{2}\rfloor}\binom{n}{2i+1}$, 
if $M=\PODI_n$; 

\item $\sum_{i=2}^{n}\sum_{k=i}^{n}\binom{n}{k}\binom{k+i-1}{2i-1}\partial_i =
\sum_{i=3}^{n}\sum_{k=i}^{n}\binom{n}{k}\binom{k+i-1}{2i-1}\partial_i =
\sum_{i=1}^{\lfloor\frac{n-1}{2}\rfloor}\sum_{k=2i+1}^{n}\binom{n}{k}\binom{k+2i}{4i+1}2^i$, 
if $M=\POD_n$. 
\end{enumerate}

Regarding the number of endomorphisms of $M\in\{\PODI_n,\POD_n\}$ satisfying Property 3 of Theorem \ref{mainthOD}, 
clearly, it coincides with the number of non-idempotent group elements of $M$, which coincides with the number of idempotents of $T$ with rank greater than or equal to two. Therefore, the number of endomorphisms of $M$ satisfying Property 3 of Theorem \ref{mainthOD} is 
\begin{enumerate}
\item $|E(\POI_n)|-|E(J_1^{\POI_n})|-|E(J_0^{\POI_n})|=2^n-n-1$, if $M=\PODI_n$; 

\item $|E(\PO_n)|-|E(J_1^{\PO_n})|-|E(J_0^{\PO_n})| = |E(\PO_n)| - \sum_{k=1}^{n}\binom{n}{k}\binom{k}{1} -1 = 
1+(\sqrt{5})^{n-1}\left(\left(\frac{\sqrt{5}+1}{2}\right)^n -\left(\frac{\sqrt{5}-1}{2}\right)^n\right) - n2^{n-1}-1=
(\sqrt{5})^{n-1}\left(\left(\frac{\sqrt{5}+1}{2}\right)^n -\left(\frac{\sqrt{5}-1}{2}\right)^n\right) - n2^{n-1}$ 
(see \cite{Laradji&Umar:2004}), if $M=\POD_n$. 
\end{enumerate}

Thus, by applying also Theorem \ref{endon} and Theorem \ref{mainthOD}, we obtain:
\begin{enumerate}
\item $|\endo(\OD_n)|=|\endo(\O_n)|+\sum_{i=1}^{\lfloor\frac{n-1}{2}\rfloor}\binom{n+2i}{4i+1}$;

\item $|\endo(\PODI_n)|=|\endo(\POI_n)|-n! + 
\lfloor\frac{n}{2}\rfloor\hskip-.1em!\, 2^{\lfloor\frac{n}{2}\rfloor} +
(2^n-n-1) + \sum_{i=1}^{\lfloor\frac{n-1}{2}\rfloor}\binom{n}{2i+1}$; 

\item $|\endo(\POD_n)|=|\endo(\PO_n)|-n! + 
\lfloor\frac{n}{2}\rfloor\hskip-.1em!\, 2^{\lfloor\frac{n}{2}\rfloor} +
\left((\sqrt{5})^{n-1}\left(\left(\frac{\sqrt{5}+1}{2}\right)^n -\left(\frac{\sqrt{5}-1}{2}\right)^n\right) - n2^{n-1}\right)+ \\
\sum_{i=1}^{\lfloor\frac{n-1}{2}\rfloor}\sum_{k=2i+1}^{n}\binom{n}{k}\binom{k+2i}{4i+1}2^i$.
\end{enumerate}

Finally, in view of  Corollary \ref{sendon}, we have: 

\begin{corollary} \label{unmOD}
Let $n\geqslant2$. Then: 
\begin{enumerate}
\item the semigroup $\OD_n$ has 
$$\textstyle
2+\sum_{i=0}^{n-1}\binom{n+i}{2i+1}F_{2i+2} + \sum_{i=1}^{\lfloor\frac{n-1}{2}\rfloor}\binom{n+2i}{4i+1}
$$ 
endomorphisms, where $F_{2i+2}$ denotes the $(2i+2)$th Fibonacci number;  

\item the semigroup $\PODI_n$ has 
$$\textstyle
1+2^n+3^n-n+\lfloor\frac{n}{2}\rfloor\hskip-.1em!\, 2^{\lfloor\frac{n}{2}\rfloor}+ \sum_{i=1}^{\lfloor\frac{n-1}{2}\rfloor}\binom{n}{2i+1}
$$
endomorphisms; 

\item the semigroup $\POD_n$ has 
\begin{multline*}\textstyle
3 + \lfloor\frac{n}{2}\rfloor\hskip-.1em!\, 2^{\lfloor\frac{n}{2}\rfloor} - n2^{n-1}+ 
2(\sqrt{5})^{n-1}\left(\left(\frac{\sqrt{5}+1}{2}\right)^n -\left(\frac{\sqrt{5}-1}{2}\right)^n\right) +  \\ \textstyle
\sum_{k=1}^{n} (\sqrt{5})^{k-1}\left(\left(\frac{\sqrt{5}+1}{2}\right)^k 
-\left(\frac{\sqrt{5}-1}{2}\right)^k\right)\sum_{i=k}^{n}\binom{n}{i}\binom{i+k-1}{2k-1}+
\sum_{i=1}^{\lfloor\frac{n-1}{2}\rfloor}\sum_{k=2i+1}^{n}\binom{n}{k}\binom{k+2i}{4i+1}2^i
\end{multline*}
endomorphisms. 
\end{enumerate}
\end{corollary}

%%%%%% 

\lastpage


\begin{thebibliography}{77}

\bibitem{Aizenstat:1962}
A.Ya. A\u{\i}zen\v{s}tat, 
The defining relations of the endomorphism semigroup of a finite linearly ordered set, 
Sibirsk. Mat. 3 (1962), 161--169 (Russian). 

\bibitem{Aizenstat:1962b} 
A.Ya. A\u{\i}zen\v{s}tat, 
Homomorphisms of semigroups of endomorphisms of ordered sets, 
Uch. Zap., Leningr. Gos. Pedagog. Inst.  238 (1962), 38--48 (Russian). 

\bibitem{Araujo&etal:2011} 
J. Ara\'ujo, V.H. Fernandes, M.M. Jesus, V. Maltcev and J.D. Mitchell, 
Automorphisms of partial endomorphism semigroups, 
Publicationes Mathematicae Debrecen 79 (2011), 23--39.

\bibitem{Cowan&Reilly:1995} 
D.F. Cowan and N.R. Reilly, 
Partial cross-sections of symmetric inverse semigroups, 
Internat. J. Algebra Comput. 5 (1995), 259--287.

%\bibitem{Delgado&Fernandes:2000}
%M. Delgado and V.H. Fernandes,
%Abelian kernels of some monoids of injective partial transformations and an application,
%Semigroup Forum 61 (2000), 435--452.

\bibitem{Delgado&Fernandes:2004}
M. Delgado and V.H. Fernandes,
Abelian kernels of monoids of order-preserving maps and of some of its extensions, 
Semigroup Forum 68 (2004), 335--356.

\bibitem{Derech:1991} 
V.D. Derech, 
On quasi-orders over certain inverse semigroups, 
Sov. Math. 35 (1991), 74--76; 
translation from Izv. Vyssh. Uchebn. Zaved. Mat. 3 (346) (1991), 76--78 (Russian).

\bibitem{Dimitrova&Koppitz:2008} 
I. Dimitrova and J. Koppitz,  
On the Maximal Subsemigroups of Some Transformation Semigroups, 
Asian-Eur. J. Math. 1 (2008), 189--202. 

\bibitem{Dimitrova&Koppitz:2009} 
I. Dimitrova and J. Koppitz, 
The Maximal Subsemigroups of the Ideals of Some Semigroups of Partial Injections, 
Discuss. Math. Gen. Algebra Appl. 29 (2009), 153--167.

\bibitem{Fernandes:1997}
V.H. Fernandes, 
Semigroups of order-preserving mappings on a finite chain: a new class of divisors, 
Semigroup Forum 54 (1997), 230--236. 

\bibitem{Fernandes:1998} 
V.H. Fernandes, 
Normally ordered inverse semigroups, 
Semigroup Forum 56 (1998), 418--433. 

\bibitem{Fernandes:2000}
V.H. Fernandes, 
The monoid of all injective orientation preserving partial transformations on a finite chain, 
Commun. Algebra 28 (2000), 3401--3426.

\bibitem{Fernandes:2001}
V.H. Fernandes, 
The monoid of all injective order-preserving partial transformations on a finite chain, 
Semigroup Forum 62 (2001), 178-204. 

\bibitem{Fernandes:2002}
V.H. Fernandes, 
Semigroups of order-preserving mappings on a finite chain: another class of divisors, 
Izv. Vyssh. Uchebn. Zaved. Mat. 3 (478) (2002), 51--59 (Russian).

\bibitem{Fernandes:2002survey}
V.H. Fernandes, 
Presentations for some monoids of partial transformations on a finite chain: a survey, 
Semigroups, Algorithms, Automata and Languages, 
eds. Gracinda M. S. Gomes \& Jean-\'Eric Pin \& Pedro V. Silva, 
World Scientific (2002), 363--378.

\bibitem{Fernandes:2008}
V.H. Fernandes,   
On divisors of pseudovarieties generated by some classes of full transformation semigroups, 
Algebra Colloq.. 15 (2008), 581--588.

\bibitem{Fernandes&Gomes&Jesus:2004}
V.H. Fernandes, G.M.S. Gomes and M.M. Jesus,
Presentations for some monoids of injective partial transformations on a finite chain,
Southeast Asian Bull. Math. 28 (2004), 903--918.

\bibitem{Fernandes&Gomes&Jesus:2005}
V.H. Fernandes, G.M.S. Gomes and M.M. Jesus, 
Congruences on monoids of order-preserving or order-reversing transformations on a finite chain,  
Glasg. Math. J.. 47 (2005), 413--424.

\bibitem{Fernandes&Gomes&Jesus:2005b}
V.H. Fernandes, G.M.S. Gomes and M.M. Jesus, 
Presentations for some monoids of partial transformations on a finite chain,
Commun. Algebra 33 (2005), 587--604.

\bibitem{Fernandes&Gomes&Jesus:2011}
V.H. Fernandes, G.M.S. Gomes and M.M. Jesus,
The cardinal and the idempotent number of various monoids of transformations on a finite chain,
Bull. Malays. Math. Sci. Soc. 34 (2011), 79-85.

\bibitem{Fernandes&al:2010} 
V.H. Fernandes, M.M. Jesus, V. Maltcev and J.D. Mitchell, 
Endomorphisms of the semigroup of order-preserving mappings, 
Semigroup Forum 81 (2010), 277--285. 

\bibitem{Fernandes&Quinteiro:2012}
V.H. Fernandes and T.M. Quinteiro, 
The cardinal of various monoids of transformations that preserve a uniform partition,
Bull. Malays. Math. Sci. Soc. 35 (2012), 885--896.

\bibitem{Fernandes&Quinteiro:2016}
V.H. Fernandes and T.M. Quinteiro, 
A note on bilateral semidirect product decompositions of some monoids of order-preserving partial permutations, 
Bull. Korean Math. Soc.. 53. (2016), 495-506.

%\bibitem{Fernandes&Volkov:2010}
%V.H. Fernandes and M. V. Volkov. 
%On divisors of semigroups of order-preserving mappings of a finite chain,
%Semigroup Forum 81 (2010), 551--554.

\bibitem{Fernandes&al:2019} 
V.H. Fernandes and P.G. Santos, 
Endomorphisms of semigroups of order-preserving partial transformations, 
Semigroup Forum 99 (2019), 333--344.

%\bibitem{Ganyushkin&Mazorchuk:2003} 
%O. Ganyushkin and V. Mazorchuk, 
%On the structure of $IO_n$, 
%Semigroup Forum 66 (2003), 455--483.

%\bibitem{Garba:1994}
%G.U. Garba,
%Nilpotents in semigroups of partial one-to-one order-preserving mappings, 
%Semigroup Forum 48 (1994), 37-49.

\bibitem{Gomes&Howie:1992}
G.M.S. Gomes and J.M. Howie, 
On the ranks of certain semigroups of order-preserving transformations, 
Semigroup Forum 45 (1992), 272--282.

\bibitem{Gyudzhenov&Dimitrova:2006} 
I. Gyudzhenov and I. Dimitrova, 
On the Maximal Subsemigroups of the Semigroup of All Monotone Transformations, 
Discuss. Math. Gen. Algebra Appl. 26 (2006), 199--217.   

%\bibitem{Higgins:1995}
%P.M. Higgins, 
%Divisors of semigroups of order-preserving mappings on a finite chain, 
%Internat. J. Algebra Comput. 5 (1995), 725--742.  

\bibitem{Howie:1971}
J.M. Howie, 
Product of idempotents in certain semigroups of transformations, 
Proc. Edinburgh Math. Soc. 17 (1971), 223--236. 

\bibitem{Howie:1995} 
J.M. Howie, 
Fundamentals of Semigroup Theory, 
Oxford, Oxford University Press, 1995.

\bibitem{Laradji&Umar:2004}
A. Laradji and A. Umar,
Combinatorial results for semigroups of order-preserving partial transformations, 
J. Algebra 278 (2004), 342--359.

\bibitem{Laradji&Umar:2006}
A. Laradji and A. Umar,
Combinatorial results for semigroups of order-preserving full transformations, 
Semigroup Forum 72 (2006), 51--62. 

\bibitem{Lavers&Solomon:1999}
T. Lavers and A. Solomon, 
The endomorphisms of a finite chain form a Rees congruence semigroup, 
Semigroup Forum 59 (1999), 167--170.

\bibitem{Ledermann:1976} 
W. Ledermann, 
Introduction to group theory, 
Longman, 1976.

\bibitem{Mazorchuk:2002}
V. Mazorchuk,
Endomorphisms of $\mathscr{B}_n , \mathscr{PB}_n$, and $\mathscr{C}_n$, 
Commun. Algebra 30 (2002), 3489--3513. 

%\bibitem{Popova:1962}
%L.M. Popova, 
%Defining relations of a semigroup of partial endomorphisms of a finite linearly ordered set,
%Leningrad. Gos. Ped. Inst. U\v cen. Zap. 238 (1962), 78--88 (Russian). 

\bibitem{Schein&Teclezghi:1997}
B.M. Schein and B. Teclezghi,  
Endomorphisms of finite symmetric inverse semigroups, 
J. Algebra 198 (1997), 300--310.  

\bibitem{Schein&Teclezghi:1998}
B.M. Schein and B. Teclezghi,   
Endomorphisms of finite full transformation semigroups, 
Proc. Am. Math. Soc. 126 (1998), 2579--2587. 

\bibitem{Sullivan:1975} 
R.P. Sullivan, 
Automorphisms of transformation semigroups, 
J. Austral. Math. Soc. 20 (Series A) (1975), 77--84.  

%\bibitem{Vernitskii&Volkov:1995}
%A.S. Vernitskii and M.V. Volkov, 
%A proof and a generalisation of Higgins' division theorem for semigroups of order preserving mappings,
%Izv. Vyssh. Uchebn. Zaved. Mat 1 (1995), 38--44

\end{thebibliography}
\end{document}